\documentclass{article}

\usepackage{amsmath,amsfonts,amssymb,url,theorem}

\usepackage[utf8]{inputenc}

{\theorembodyfont{\slshape}
\newtheorem{proposition}{Proposition}[section]
\newtheorem{lemma}[proposition]{Lemma}
\newtheorem{corollary}[proposition]{Corollary}

\newtheorem{theorem}[proposition]{Theorem}}

{\theorembodyfont{\upshape}

}

\newcommand{\eps}{\varepsilon}

\newcommand\Q{{\mathbb Q}}

\newcommand\Z{{\mathbb Z}}

\newcommand{\CC}{{\mathcal{C}}}

\newcommand{\OO}{{\mathcal{O}}}

\newcommand\qed{\hfill$\square$}

\renewcommand{\Im}{{\mathrm{Im}}}
\newcommand{\Gal}{{\mathrm{Gal}}}

\newcommand{\ord}\nu

\newenvironment{proof}{\paragraph{Proof}}{}

\title{Fields generated by sums and products of singular moduli}

\author{Bernadette FAYE, Antonin RIFFAUT}

1

\makeatletter

\renewcommand*\l@section[2]{%
  \ifnum \c@tocdepth >\z@
    \addpenalty\@secpenalty
    \addvspace{0.2em \@plus\p@}%
    \setlength\@tempdima{1.5em}%
    \begingroup
      \parindent \z@ \rightskip \@pnumwidth
      \parfillskip -\@pnumwidth
      \leavevmode \bfseries
      \advance\leftskip\@tempdima
      \hskip -\leftskip
      #1\nobreak\hfil \nobreak\hb@xt@\@pnumwidth{\hss #2}\par
    \endgroup
  \fi}
  
\makeatother

\begin{document}

\hfuzz 5pt

\maketitle

\begin{abstract}
We show that the field  $\Q(x,y)$, generated by two singular moduli~$x$ and~$y$, is generated by their sum ${x+y}$,   unless~$x$ and~$y$ are conjugate over~$\Q$, in which case ${x+y}$ generates a subfield of degree at most~$2$. We obtain a similar result for the product of two singular moduli. 
\end{abstract}

{\footnotesize 

\tableofcontents

}

\section{Introduction}
A \textsl{singular modulus} is the $j$-invariant of an elliptic curve with complex multiplication. Given a singular modulus~$x$ we denote by $\Delta_x$ the discriminant of the associated imaginary quadratic order. 
We denote by $h(\Delta)$ the class number of the imaginary quadratic order of discriminant~$\Delta$. Recall that two singular moduli~$x$ and~$y$ are conjugate over~$\Q$ if and only if ${\Delta_x=\Delta_y}$, and that all singular moduli of a given discriminant~$\Delta$ form a full Galois orbit over~$\Q$. In particular, ${[\Q(x):\Q]=h(\Delta_x)}$. For all details, see, for instance, \cite[\S 7 and \S 11]{Co13}

Starting from the ground-breaking article of André~\cite{An98} equations involving singular moduli were studied by many authors, see~\cite{ABP15,BLP16,Ri17} for a historical account and further references. In particular, Kühne~\cite{Ku13} proved that equation ${x+y=1}$ has no solutions in singular moduli~$x$ and~$y$, and Bilu et al~\cite{BMZ13} proved the same for the equation ${xy=1}$.
These results where generalized in~\cite{ABP15} and~\cite{BLP16}. 

\begin{theorem} {\rm \cite{ABP15,BLP16}}
\label{thspinq}
Let~$x$ and~$y$ be singular moduli such that ${x+y\in \Q}$ or ${xy\in \Q^\times}$. Then either ${h(\Delta_x)=h(\Delta_y)=1}$ or ${\Delta_x=\Delta_y}$ and ${h(\Delta_x)=h(\Delta_y)= 2}$. 
\end{theorem}
Here the statement about ${x+y}$ is (a special case of) Theorem~1.2 from~\cite{ABP15}, and the statement about $xy$ is Theorem~1.1 from~\cite{BLP16}. 

Note that lists of all imaginary quadratic discriminants~$\Delta$ with ${h(\Delta)\le 2}$ are widely available, so Theorem~\ref{thspinq} is fully explicit. 

In view of Theorem~\ref{thspinq} one may ask the following question: how much does the number field generated by the sum ${x+y}$ or the product $xy$ of two singular moduli differ from the field $\Q(x,y)$? The objective of this note is to show that the fields ${\Q(x+y)}$ and $\Q(xy)$ (provided ${xy\ne 0}$) are subfields of $\Q(x,y)$ of degree at most~$2$, and in ``most cases'' each of ${x+y}$ and $xy$ generates $\Q(x,y)$. Here are our principal results.  

\begin{theorem}
\label{thsum}
Let~$x$ and~$y$ be  singular moduli. Then ${\Q(x+y)=\Q(x,y)}$ if ${\Delta_x\ne \Delta_y}$, and  ${[\Q(x,y):\Q(x+y)]\le 2}$ if ${\Delta_x=\Delta_y}$. 
\end{theorem}

\begin{theorem}
\label{thprod}
Let~$x$ and~$y$ be non-zero singular moduli. Then ${\Q(xy)=\Q(x,y)}$ if ${\Delta_x\ne \Delta_y}$, and  ${[\Q(x,y):\Q(xy)]\le 2}$ if ${\Delta_x=\Delta_y}$. 
\end{theorem}

Both the ``sum'' and the ``product'' statements of Theorem~\ref{thspinq} are very special cases of these two theorems. 

Note that in the case ${\Delta_x=\Delta_y}$, the statements ${[\mathbb{Q}(x,y):\mathbb{Q}(x+y)]\leq 2}$ and ${[\mathbb{Q}(x,y):\mathbb{Q}(xy)]\leq 2}$ are best possible: one cannot expect that ${x+y}$ or $xy$ always generates $\Q(x,y)$ in this case. Indeed, for instance,  when~$x$ is a non-real singular modulus and ${y=\bar x}$ is the complex conjugate of~$x$, then neither  ${x+y}$ nor $xy$ generates $\Q(x,y)$. 

Theorems~\ref{thsum} and~\ref{thprod} are proved, respectively, in Sections~\ref{ssum} and~\ref{sprod} after some preparations in Section~\ref{sprem} and~\ref{sfields}.

\vspace{1\baselineskip}

\noindent \textbf{Acknowledgements.} We thank Florian Luca for useful discussions and ideas. This paper was written during a visit of B. Faye at the Institut Math\'ematiques de Bordeaux in Fall 2017. She thanks Yuri Bilu for his invitation and mentoring. During the preparation of this paper, B. Faye was also supported by the EMS-CDC, the IRN GandA (CNRS) and the ALGANT Program. 

Our calculations were performed using the \textsf{PARI/GP} package \cite{pari}. The sources are available from the second author.



\section{Preliminaries}
\label{sprem}
Everywhere below the letter~$\Delta$ stands for an imaginary quadratic discriminant, that is, ${\Delta<0}$ and satisfies ${\Delta\equiv 0,1\bmod 4}$. The letter~$D$ will denote a fundamental discriminant; that is, in addition to the two conditions imposed on~$\Delta$, when ${D\equiv 0\bmod 4}$ we have ${D/4\equiv 2,3\bmod 4}$. 

We denote by $\OO_\Delta$ the imaginary quadratic order of discriminant~$\Delta$, that is, ${\OO_\Delta =\Z[(\Delta+\sqrt\Delta)/2]}$. Then ${\Delta=Df^2}$, where~$D$ is  discriminant of the number field ${K=\Q(\sqrt\Delta)}$ and ${f=[\OO_D:\OO_\Delta]}$ is the conductor. 

We denote by $\CC(\Delta)$ and by $h(\Delta)$ the class group and the class  number of~$\OO_\Delta$,  so that ${h(\Delta) =\#\CC(\Delta)}$. 

Given a singular modulus~$x$, we write ${\Delta_x=D_xf_x^2}$ with~$D_x$ the fundamental discriminant and~$f_x$ the conductor. We denote by~$\tau_x$ the only~$\tau$ is the standard fundamental domain such that ${j(\tau)=x}$. Further, we denote by~$K_x$ the associated imaginary quadratic field:
$$
K_x=\Q(\tau_x)=\Q(\sqrt{D_x})=\Q(\sqrt{\Delta_x}). 
$$

Recall the following basic properties.  
\begin{itemize}
\item
The singular moduli of discriminant~$\Delta$ form a full Galois orbit over~$\Q$ and over $\Q(\sqrt\Delta)$ as well. In particular, singular moduli~$x$ and~$y$ are conjugate over~$\Q$ if and only if ${\Delta_x=\Delta_y}$.


\item
There is a one-to-one correspondence between the singular moduli of discriminant~$\Delta$ and the set~$T_\Delta$ of triples $(a,b,c)$ of integers satisfying ${b^2-4ac=\Delta}$ and some other  conditions; see, for instance, \cite[Proposition~2.5]{BLP16}.  If  ${(a,b,c)\in T_\Delta}$ then ${(b+\sqrt{\Delta})/2}$ belongs to the standard fundamental domain, and the corresponding singular modulus is 
${j((b+\sqrt{\Delta})/2a)}$. 

\item
We say that a singular modulus is \textsl{dominant} if in the corresponding triple $(a,b,c)$ we have ${a=1}$, and \textsl{subdominant} if ${a=2}$. There exists exactly one dominant and at most two subdominant singular moduli of a given discriminant~$\Delta$, see \cite[Proposition 2.6]{BLP16}. 

\end{itemize}

We will systematically use the inequality 
$$
\bigl||j(z)|-e^{2\pi \Im z}\bigr|\le 2079,
$$
\cite[Lemma~1]{BMZ13}, 
which holds true for every~$z$ in the standard fundamental domain. In particular, if~$x$ is a singular modulus of discriminant~$\Delta$ corresponding to the triple ${(a,b,c)\in T_{\Delta}}$  then 
$$
\bigl||x|-e^{\pi|\Delta_x|^{1/2}/a}\bigr|\le 2079.
$$

\section{Fields generated by singular moduli}
\label{sfields}
Let~$G$ be a finite  group. We say that~$G$ is a group of \textsl{dihedral type} if there exists an abelian subgroup ${H<G}$ of index~$2$ and an element ${\iota\in G}$ of order~$2$ such that for any ${g\in H}$ we have ${\iota g\iota=g^{-1}}$. We call the couple $(H,\iota)$ the \textsl{dihedral structure} on~$G$. 

Note that a group of dihedral type may be abelian; in this case it is $2$-elementary (that is, isomorphic to ${\Z/2\Z\times\cdots\times\Z/2\Z}$).

The following simple lemma can be found in~\cite{BKR17}, where it is credited to Lenstra. Since the article~\cite{BKR17} did not appear yet, we include a short proof for the reader's convenience. 

\begin{lemma}
\label{llen}
Let~$G$ be a non-abelian group of dihedral type with dihedral structure $(H,\iota)$. Then~$H$ is generated by all elements of~$G$ of order $>2$. In particular,~$H$ is unique: if $(H',\iota')$ is another dihedral structures on~$G$, then ${H=H'}$. 
\end{lemma}

\begin{proof}
Note first of all~$G$ does contain elements of order $>2$ because it is not abelian. All of them must belong to~$H$, because every element of ${G\smallsetminus H}$ is of the form $\iota h$ with ${h\in H}$; hence it is of order~$2$. It remains to show that every element of~$H$ of order~$2$ is product of two elements of bigger order. Let ${k\in H}$ be of order $>2$ and let ${h\in H}$ be of order~$2$. Since~$H$ is abelian, $kh$ is also of order $>2$. Writing ${h=k^{-1}\cdot kh}$, we are done.   \qed
\end{proof}

\bigskip

Now let~$x$ be a singular modulus. We write ${\Delta=\Delta_x}$ and ${K=K_x}$. It is known that 
the field $K(x)$ is Galois over~$\Q$.  Set ${G=\Gal(K(x)/\Q)}$, ${H=\Gal(K(x)/K)}$ and let ${\iota\in G}$ be the complex conjugation.  It is known that~$H$ is isomorphic to $\CC(\Delta)$.

The following is well-known: see, for instance \cite[Lemma~9.3]{Co13} and   \cite[Corollary~3.3]{ABP15}. 

\begin{proposition}
\label{pcox}
The group~$G$ is of dihedral type, with dihedral structure $(H,\iota)$. Furthermore, the following properties are equivalent.

\begin{enumerate}
\item
\label{iab}
The group~$G$ is abelian. 

\item
\label{i2}
The group~$H$ is $2$-elementary. 

\item
\label{i22}
The group~$G$ is $2$-elementary. 

\item
\label{igal}
The field $\Q(x)$ is Galois over~$\Q$.

\item
\label{iab1}
The field $\Q(x)$ is abelian over~$\Q$.

\end{enumerate}

\end{proposition}


\begin{corollary}
\label{ctwoconj}
Let $x,x',y,y'$ be singular moduli. Assume that 
$$
\Delta_x=\Delta_{x'}, \quad \Delta_y=\Delta_{y'}, \quad D_x\ne D_y. 
$$
Assume further that ${\Q(x,x')=\Q(y,y')}$. Then ${\Q(x)=\Q(y)}$. 
\end{corollary}

\begin{proof}
If $\Q(x)$ is Galois over~$\Q$ then ${\Q(x,x')=\Q(x)}$ and $\Gal(\Q(x)/\Q)$ is $2$-elementary by Proposition~\ref{pcox}. Then $\Gal(\Q(y,y')/\Q)$ is  $2$-elementary as well, which implies that $\Q(y)$ is Galois over~$\Q$, which implies that ${\Q(y)=\Q(y,y')}$. Hence ${\Q(x)=\Q(y)}$.

Now assume that $\Q(x)$ is not Galois over~$\Q$. We will see that this leads to a contradiction.  Since $K_x(x)$ is Galois over~$\Q$ and ${[K_x(x):\Q(x)]\le 2}$, the field $K_x(x)$ is the Galois closure of $\Q(x)$ over~$\Q$.

Denote by~$L$ the Galois closure of $\Q(x,x')$ over~$\Q$. Since 
$$
\Q(x)\subset \Q(x,x')\subset K_x(x), 
$$
we have ${L=K_x(x)}$. Since ${\Q(x,x')=\Q(y,y')}$, we have similarly ${L=K_y(y)}$.

Set ${H_x=\Gal(L/K_x)}$ and ${H_y=\Gal(L/K_y)}$, and let~$\iota$ be the complex conjugation. Then $(H_x,\iota)$ and $(H_y,\iota)$ are dihedral structures in ${G=\Gal(L/\Q)}$. Proposition~\ref{pcox} implies that~$G$ is not abelian, and Lemma~\ref{llen} implies that ${H_x=H_y}$. Hence ${K_x=K_y}$, contradicting the assumption ${D_x\ne D_y}$. The proof is complete. \qed
\end{proof}

\section{Proof of Theorem~\ref{thsum}}
\label{ssum}
Let~$x$ and~$y$ be singular moduli. We want to show that   ${\Q(x,y)=\Q(x+y)}$ if ${\Delta_x\ne \Delta_y}$  and  ${[\Q(x,y):\Q(x+y)]\le 2}$ if ${\Delta_x=\Delta_y}$. We may clearly assume that ${x\ne y}$. In this case we will prove the following more general statement.

{\sloppy

\begin{theorem}
Let~$x$ and~$y$ be distinct singular moduli and ${\eps\in\{\pm1\}}$. Then ${\Q(x,y)=\Q(x+\eps y)}$, unless ${\eps=1}$ and ${\Delta_x=\Delta_y}$, in which case we have ${[\Q(x,y):\Q(x+y)]\le 2}$. 
\end{theorem}

}

Let~$L$ be the Galois closure of $\Q(x,y)$ over~$\Q$. Set 
$$
G=\Gal(L/\Q(x+\eps y)), \qquad H=\Gal(L/\Q(x,y)).
$$
Note that 
$$
H=\{\sigma \in G: x^\sigma=x\}=\{\sigma \in G: y^\sigma=y\}. 
$$
We want to show that ${G=H}$, unless
\begin{equation}
\label{except}
\Delta_x=\Delta_y,\qquad \eps=1, 
\end{equation}
in which case ${[G:H]\le 2 }$.  

\subsection{Equal discriminants}
\label{sseq}
We start from the case 
${\Delta_x=\Delta_y=\Delta}$. 
We may assume that~$x$ is dominant and~$y$ is not (recall that ${x\ne y}$. It follows that
$$
|x|\ge e^{\pi|\Delta|^{1/2}}-2079, \qquad |y|\le e^{\pi|\Delta|^{1/2}/2}+2079. 
$$
and 
\begin{equation}
\label{elowerequal}
|x+\eps y|\ge e^{\pi|\Delta|^{1/2}}-e^{\pi|\Delta|^{1/2}/2}-4158. 
\end{equation}

\subsubsection{The case ${\eps=1}$}
Assume  that ${\eps =1}$ and let us prove that ${[G:H]\le 2}$.  If ${[G:H]>2}$ then there exists ${\sigma \in G}$ such that ${x^\sigma \ne x}$ and ${x^\sigma \ne y}$. Since ${x+ y=x^\sigma+ y^\sigma }$ by the definition of the group~$G$, we also have ${y^\sigma \ne x}$. Thus, neither~$x^\sigma$ nor~$y^\sigma$ is dominant. It follows that 
\begin{equation*}
|x^\sigma|\le e^{\pi|\Delta|^{1/2}/2}+2079, \qquad |y^\sigma|\le e^{\pi|\Delta|^{1/2}/2}+2079, 
\end{equation*}
and 
\begin{equation}
\label{eupperequal}
|x+y|=|x^\sigma+y^\sigma|\le 2e^{\pi|\Delta|^{1/2}/2}+4158.
\end{equation}
This contradicts~\eqref{elowerequal} when~$|\Delta|\geq 9$, and for ${|\Delta|\leq 8}$ we have ${h(\Delta)=1}$, and so ${G=H}$ is a trivial group. 

\subsubsection{The case ${\eps=-1}$}
Now assume that ${\eps=-1}$ and let us prove that ${G=H}$. If  ${G\ne H}$ then there exist ${\sigma \in G}$ such that ${x^\sigma \ne x}$. Then ${y^\sigma \ne x}$ either; in the opposite case we would have ${2x=x^\sigma+y}$, which is impossible, since~$x$ is dominant, but~$x^\sigma$ and~$y$ are not. Thus, neither~$x^\sigma$ nor~$y^\sigma$ is dominant, and we again obtain~\eqref{eupperequal}, which together with~\eqref{elowerequal} implies  that ${|\Delta|\leq 8}$, in which case ${G=H}$ is a trivial group.  


\subsection{Equal fundamental discriminants}
Now let us assume that ${D_x=D_y=D}$, but ${f_x\ne f_y}$. We may assume that ${f_x>f_y}$ and that~$x$ is dominant.

Since  ${f_x> f_y}$ we have  ${f_x\ge f_y+1}$, and 
$$
|\Delta_y|^{1/2}=|D|^{1/2}f_y\le |D|^{1/2}f_x-|D|^{1/2}\le |\Delta_x|^{1/2}-\sqrt3. 
$$
Hence 
$$
|x|\ge e^{\pi|\Delta_x|^{1/2}}-2079, \qquad  |y|\le e^{\pi|\Delta_x|^{1/2}-\pi\sqrt3}+2079 \le 0.01 e^{\pi|\Delta_x|^{1/2}}+2079
$$
and 
\begin{equation}
\label{elowerdistinct}
|x+\eps y|\ge 0.99e^{\pi|\Delta_x|^{1/2}}-4158. 
\end{equation} 

Now let us assume that ${H\ne G}$. Then there exists ${\sigma\in G}$ such that  ${x^\sigma\ne x}$; in particular,~$x^\sigma$ is not dominant, and we have 
$$
|x^\sigma|\le e^{\pi|\Delta_x|^{1/2}/2}-2079, \qquad  |y^\sigma|\le e^{\pi|\Delta_x|^{1/2}-\pi\sqrt3}+2079 \le 0.01 e^{\pi|\Delta_x|^{1/2}}+2079. 
$$
Therefore 
$$
|x+\eps y|=|x^\sigma+y^\sigma|\le 0.01 e^{\pi|\Delta_x|^{1/2}}+ e^{\pi|\Delta_x|^{1/2}/2}+4158. 
$$
This contradicts~\eqref{elowerdistinct} when $|\Delta_x|\geq 9$, and for  ${|\Delta_x|\leq 8}$  the group ${G=H}$ is trivial.  

\subsection{Distinct fundamental discriminants}
\label{ssdist}
Now assume that ${D_x\ne D_y}$. 
If ${G\ne H}$ then there exists ${\sigma \in G}$ such that ${x^\sigma \ne x}$. For such~$\sigma$ we have 
${x-x^\sigma =\eps(y^\sigma-y)}$. 
In particular, 
${\Q(x-x^\sigma)=\Q(y-y^\sigma)}$. 

{\sloppy

As we have seen in Subsection~\ref{sseq},
${\Q(x-x^\sigma)=\Q(x,x^\sigma)}$. Similarly, ${\Q(y-y^\sigma)=\Q(y,y^\sigma)}$.  
We obtain  ${\Q(x,x^\sigma)=\Q(y,y^\sigma)}$. 

}

Now Corollary~\ref{ctwoconj} implies that ${\Q(x)=\Q(y)}$. Hence that our~$\Delta_x$ and~$\Delta_y$ are listed in Table~2 on page~12 of~\cite{ABP15}.  For these values of~$\Delta_x$ and~$\Delta_y$ the statement can be verified directly using \textsf{PARI}~\cite{pari}.  

\section{Proof of Theorem~\ref{thprod}}
\label{sprod}
It is rather similar to the proof of Theorem~\ref{thsum}, though technically a bit more complicated. Let~$x$ and~$y$ be non-zero singular moduli. We want to show that   ${\Q(x,y)=\Q(xy)}$ if ${\Delta_x\ne \Delta_y}$  and  ${[\Q(x,y):\Q(xy)]\le 2}$ if ${\Delta_x=\Delta_y}$. Since ${\Q(x^2)=\Q(x)}$ \cite[Lemma 2.6]{Ri17}, we may  assume that ${x\ne y}$. In this case we will prove the following more general statement. 

\begin{theorem}
Let~$x$ and~$y$ be distinct non-zero singular moduli and ${\eps\in\{\pm1\}}$. Then ${\Q(x,y)=\Q(xy^\eps)}$, unless ${\eps=1}$ and ${\Delta_x=\Delta_y}$, in which case we have ${[\Q(x,y):\Q(xy)]\le 2}$. 
\end{theorem}

We again denote by~$L$  the Galois closure of $\Q(x,y)$ over~$\Q$, and we set 
$$
G=\Gal(L/\Q(xy^\eps)), \qquad H=\Gal(L/\Q(x,y)).
$$
We want to show that ${G=H}$, unless~\eqref{except} holds, 
in which case ${[G:H]\le 2 }$.  

\subsection{Equal discriminants}
\label{sseqprod}
We again start from the case 
${\Delta_x=\Delta_y=\Delta}$. 
We may assume that~$x$ is dominant and~$y$ is not.

\subsubsection{The case ${\eps=1}$}
\label{ssseone}
Assume first that ${\eps=1}$ and let us prove that ${[G:H]\le 2}$.  We have the lower bound 
\begin{equation}
\label{elowereq}
|xy|\ge 3000e^{\pi|\Delta|^{1/2}}\min\{10^{-8}, |\Delta|^{-3}\}, 
\end{equation}
see~\cite{BLP16}, equation~(12). If ${[G:H]>2}$ then there exists ${\sigma \in G}$ such that ${x^\sigma \ne x}$ and ${x^\sigma \ne y}$. Since ${x y=x^\sigma y^\sigma }$ by the definition of~$G$, we also have ${y^\sigma \ne x}$. Thus, neither~$x^\sigma$ nor~$y^\sigma$ is dominant. 

Now we have two cases. If one of~$x^\sigma$,~$y^\sigma$ is not subdominant then we have the upper bound 
\begin{equation}
\label{euppereq}
|xy|=|x^\sigma y^\sigma|\le (e^{\pi|\Delta|^{1/2}/2}+2079)(e^{\pi|\Delta|^{1/2}/3}+2079), 
\end{equation}
which contradicts~\eqref{elowereq} when~$|\Delta|\geq 396$. 

Now assume that both~$x^\sigma$ and~$y^\sigma$ are subdominant. Then from  \cite[Proposition~2.6]{BLP16}, we have that  ${\Delta=1\bmod8}$ and 
$$
\{\tau_{x^\sigma}, \tau_{y^\sigma}\}=\left\{\frac{-1+\sqrt\Delta}{4},\frac{1+\sqrt\Delta}{4}\right\}.  
$$
In particular, ${\tau_{x^\sigma}-\tau_{y^\sigma}=\pm1/2}$, which implies that $(x^\sigma,y^\sigma)$ is a point on the modular curve $Y_0(4)$, defined by the equation ${\Phi_4(X,Y)=0}$, where $\Phi_4$ is the classical modular polynomial of level~$4$. Since $\Phi_4$ has coefficients in~$\Q$, the point $(x,y)$ must also belong to this curve. But we have 
\begin{align*}
\Phi_4(X,j(\tau))=&\left(X-j\left(\frac{\tau}{4}\right)\right)\left(X-j\left(\frac{\tau+1}{4}\right)\right)\left(X-j\left(\frac{\tau+2}{4}\right)\right)\times\\ 
&\left(X-j\left(\frac{\tau+3}{4}\right)\right)(X-j(4\tau))\left(X-j\left(\tau+\frac{1}{2}\right)\right).
\end{align*} 
Since ${\tau_x=(1+\sqrt\Delta)/2}$, we must have 
${\tau_y =(b+\sqrt\Delta)/8}$ with some~$b$. This shows that in this case we have a sharper, than~\eqref{elowereq} lower bound 
\begin{equation}
\label{elowermodul}
|xy|\ge (e^{\pi|\Delta|^{1/2}}-2079)(e^{\pi|\Delta|^{1/2}/4}-2079).
\end{equation}
We also have the upper bound 
\begin{equation}
\label{euppermodul}
|xy|=|x^\sigma y^\sigma|\le (e^{\pi|\Delta|^{1/2}/2}+2079)^2.  
\end{equation}
Comparing the two bounds, we again obtain a contradiction when~$|\Delta|\geq 95$.

Thus in any case ${|\Delta|\le 395}$, and  the condition ${[G:H]\le 2}$ can be verified by a direct calculation using \textsf{PARI}. 

\subsubsection{The case ${\eps=-1}$}
\label{sssepsminone}
Now assume that ${\eps=-1}$ and let us prove that ${G=H}$. If  ${G\ne H}$ then there exist ${\sigma \in G}$ such that ${x^\sigma \ne x}$. 
We obtain the equality 
${xy^\sigma=x^\sigma y}$, 
with both~$x^\sigma$ and~$y$ not dominant.  If one of them is not subdominant either, then we have bounds of the form~\eqref{elowereq} and~\eqref{euppereq} for~$|xy^\sigma|$, leading to a contradiction when $|\Delta|\geq 396$; for the small $|\Delta|$  it can be verified directly that ${G=H}$.  

If both~$x^\sigma$ and~$y$ are  subdominant, then $(x^\sigma,y)$ is a point on $Y_0(4)$. Then so is $(x,y^{\sigma^{-1}})$.  We have
${xy^{\sigma^{-1}}=x^{\sigma^{-1}} y}$ with both~$x^{\sigma^{-1}}$ and~$y$ not dominant. Arguing  as in Subsection~\ref{ssseone}, we bound ${|xy^{\sigma^{-1}}|}$ from below with the right-hand side of~\eqref{elowermodul}, and ${|x^{\sigma^{-1}} y|}$ from above by the right-hand side of~\eqref{euppermodul}, again arriving to a contradiction when $|\Delta|\geq 95$. Thus, in any case we have ${|\Delta|\le 395}$, and for these small values of~$|\Delta|$ it can be verified that ${G=H}$ using \textsf{PARI}.

\subsection{Equal fundamental discriminants}
Now let us assume that ${D_x=D_y=D}$, but ${f_x\ne f_y}$. We want to show that ${G=H}$. 
We may assume that~$x$ is dominant and that ${f_x>f_y}$. 

Assume that ${G\ne H}$. Then there exists ${\sigma\in G}$ such that ${x^\sigma \ne x}$. Then also ${y^\sigma \ne y}$. We have ${x/x^\sigma=(y/y^\sigma)^{\eps}}$, and, in particular, ${\Q(x/x^\sigma)=\Q(y/y^\sigma)}$. The result of Subsection~\ref{sssepsminone} implies that ${\Q(x,x^\sigma)=\Q(y,y^\sigma)}$. It follows that ${K(x,x^\sigma)=K(y,y^\sigma)}$, where we set ${K=K_x=K_y=\Q(\sqrt D)}$. Since the fields $K(x)$ and $K(y)$ are Galois over~$\Q$, we obtain ${K(x)=K(y)}$. Proposition 4.3  from~\cite{ABP15} now implies\footnote{To be precise, the hypothesis of Proposition 4.3  from~\cite{ABP15} is (in our notation) ${\Q(x)=\Q(y)}$, which is formally stronger than ${K(x)=K(y)}$. But in the proof of this proposition it is used only that ${K(x)=K(y)}$.} that either ${h(\Delta_x)=h(\Delta_y)=1}$, which is impossible because ${G\ne H}$, or ${f_x/f_y\in \{2,1,1/2\}}$. Since ${f_x>f_y}$, we have ${f_x=2f_y}$.

In the sequel we denote ${\Delta_y=\Delta}$ and ${\Delta_x=4\Delta}$. Note that ${\Delta\equiv 1\bmod 8}$, see \cite[Subsection~3.2.2]{BLP16}. This implies that there are no subdominant singular moduli of discriminant $4\Delta$, see \cite[Proposition~2.6]{BLP16}. In particular,~$x^\sigma$ is not subdominant. 

We have 
\begin{equation}
\label{etwo}
xy=x^\sigma y^\sigma\quad  \text{if $\eps=1$}, \qquad  xy^\sigma=x^\sigma y \quad \text{if $\eps=-1$}.
\end{equation}
Since~$x$ is dominant,  both left-hand sides in~\eqref{etwo} can be bounded from below as 
$$
3000e^{\pi|\Delta_x|^{1/2}}\min\{10^{-8}, |\Delta_y|^{-3}\} \ge  3000e^{2\pi|\Delta|^{1/2}}\min\{10^{-8}, |\Delta|^{-3}\}
$$
see~\cite{BLP16}, equation~(12).
Further, since~$x^\sigma$ is neither dominant nor subdominant, both right-hand sides in~\eqref{etwo} can be estimated from above as 
$$
(e^{\pi|\Delta_x|^{1/2}/3}+2079)(e^{\pi|\Delta_y|^{1/2}}+2079) \le (e^{2\pi|\Delta|^{1/2}/3}+2079)(e^{\pi|\Delta|^{1/2}}+2079)
$$
Comparing the lower and the upper estimates, we obtain a contradiction for $|\Delta|\geq 99$. For the remaining small values of $\Delta$ the condition ${G=H}$ can be verified directly using \textsf{PARI}.

\subsection{Distinct fundamental discriminants}
Now assume that ${D_x\ne D_y}$. We argue exactly as in Subsection~\ref{ssdist}. 
If ${G\ne H}$ then there exists ${\sigma \in G}$ such that ${x^\sigma \ne x}$. For such~$\sigma$ we have 
${x/x^\sigma =(y^\sigma/y)^{\eps}}$. 
In particular, 
${\Q(x/x^\sigma)=\Q(y/y^\sigma)}$.
The result of Subsection~\ref{sseqprod} implies that   ${\Q(x,x^\sigma)=\Q(y,y^\sigma)}$, and we obtain   ${\Q(x)=\Q(y)}$ by Corollary~\ref{ctwoconj}. Hence  our~$\Delta_x$ and~$\Delta_y$ are listed in Table~2 on page~12 of~\cite{ABP15}.  For these values of~$\Delta_x$ and~$\Delta_y$ the statement can be verified directly using \textsf{PARI}.

{\footnotesize

\bibliographystyle{amsplain}
\bibliography{sketch4}

}

\end{document}